\PassOptionsToPackage{quiet}{fontspec}
\documentclass[leqno]{article}

\usepackage{amsmath,amsthm,amsfonts,color,upgreek}
\usepackage{amssymb}
\usepackage{graphicx}
\usepackage{setspace}
\usepackage{mathrsfs}
\usepackage{graphicx}
\usepackage[shortlabels]{enumitem}
\usepackage{authblk}
\usepackage{verbatim}
\usepackage{comment}

\newtheorem{theorem}{Theorem}[section]
\newtheorem{lemma}[theorem]{Lemma}
\newtheorem{corollary}[theorem]{Corollary}
\newtheorem{remark}[theorem]{Remark}

\usepackage{chngcntr}
\numberwithin{equation}{section}

\usepackage{calc}
\usepackage{tikz}
\usepackage{setspace}
\usepackage{geometry}
\usepackage{pgffor}%可以使用foreach的包
\usepackage{ifthen}%可以使用ifthenelse的包，还能使用whiledo
\usetikzlibrary{arrows.meta}%画箭头用的包
\usetikzlibrary{decorations.markings}
\tikzstyle{vertex}=[circle, draw, inner sep=2pt, minimum size=6pt]
\tikzstyle{filledvertex}=[circle, draw, fill, inner sep=2pt, minimum size=6pt]

\newcommand{\F}{\mathcal{F}}

\title{Typical intersecting families are trivial
\thanks{School of Mathematics and Statistics, 
Xi'an--Budapest Joint Research Center for Combinatorics,
Northwestern Polytechnical University,
Xi'an 710129, China; Research and Development Institute of Northwestern Polytechnical University in Shenzhen,
Shenzhen 518057, China.} }
\author{Yandong Bai
\thanks{E-mail: {\tt bai@nwpu.edu.cn}.}
\quad Haoyun Gu
\thanks{E-mail: {\tt haoyungu@mail.nwpu.edu.cn}.}
\quad Wenston J.T. Zang
\thanks{E-mail: {\tt zang@nwpu.edu.cn}.}
}
\date{}

\begin{document}

\maketitle

\begin{abstract}
We study the counting problem for non-uniform intersecting families in extremal set theory. Let $J(n,k)$ denote the number of intersecting families $\mathcal{F}\subset 2^{[n]}$ such that every member of $\mathcal{F}$ has size at most $k$. Extending recent counting results for uniform intersecting families, we prove that for $n\ge 2k+2+2\sqrt{k \log k}$ and $k \rightarrow +\infty$, 
\[
    J(n,k)=(n+o(1))2^{\sum_{i=1}^{k}\binom{n-1}{i-1}}.
\]
This result reveals that typical non-uniform intersecting families of bounded size are trivial, i.e., almost all such families share a common fixed element. 
%The proof combines known asymptotic enumeration of uniform intersecting families with cross‑intersecting bounds, and addresses the additional complexity arising from heterogeneous set sizes through a structural decomposition based on the largest sets.
\end{abstract}



\section{Introduction}

Let $[n] = \{1,\ldots,n\}$ be the standard $n$-element set and $2^{[n]}$ its power set. A family $\mathcal{F} \subset 2^{[n]}$ is called \emph{intersecting} if $F \cap F' \neq \emptyset$ for all $F, F' \in \mathcal{F}$; and $\F$ is called \textit{trivial} if all sets in $\F$ have non-empty intersection. For an integer $k \ge 1$ we use $\binom{[n]}{k}$ and $\binom{[n]}{\le k}$ to denote the families of all subsets of size $k$ and all subsets of size at most $k$ of $[n]$, respectively. 

The seminal result of Erd\H{o}s, Ko and Rado \cite{EKR61} states that if $\mathcal{F} \subset \binom{[n]}{k}$ is intersecting and $n \ge 2k\ge 4$, then 
\[
    \#\mathcal{F} \le \binom{n-1}{k-1}.
\] 
The trivial intersecting family that consists of all $k$-sets containing a fixed element, shows that this inequality is best possible. For non-trivial intersecting families (those with empty total intersection), Hilton and Milner \cite{HM67} determined the maximum size: for $n > 2k \ge 4$, if $\mathcal{F} \subset \binom{[n]}{k}$ is intersecting and non-trivial, then
\[
\#\mathcal{F} \le \binom{n-1}{k-1} - \binom{n-k-1}{k-1} + 1,
\]
with equality characterized by the so-called Hilton–Milner families.

Call a family $\F\subset \binom{[n]}{k}$ a \textit{$k$-uniform} family.
Let $I(n,k)$ denote the number of $k$-uniform intersecting families in $\binom{[n]}{k}$ and $I(n,k,\ge 1)$ the number of non-trivial such families. 
Improving a result of Balogh, Das, Delcourt, Liu and Sharifzadeh \cite{B+15JCTA}, Frankl and Kupavskii \cite{FK18} and independently Balogh, Das, Liu, Sharifzadeh and Tran \cite{B+19} showed that if $n \ge 2k + 2 + c\sqrt{k\ln k}$ for some positive constant $c$, then typical uniform intersecting families are trivial. 

\begin{theorem}[Balogh et al. \cite{B+19}, Frankl and Kupavskii \cite{FK18}]
For $n \ge 2k + 2 + c\sqrt{k\log k}$ and $k \to \infty$, 
\[
I(n,k) = (n + o(1))2^{\binom{n-1}{k-1}} \quad \text{and} \quad I(n,k,\ge 1) = o\left(2^{\binom{n-1}{k-1}}\right),
\]
where $c = 2$ in \cite{FK18} and $c$ was a large constant in \cite{B+19}.
\end{theorem}

A slightly weaker result that covers the entire range of parameters was obtained via the graph container method in \cite{B+15JCTA}.

\begin{theorem}[Balogh et al. \cite{B+15JCTA}]
For $n \ge 2k + 1$, 
\[
I(n,k) = 2^{(1 + o(1))\binom{n-1}{k-1}}.
\]
\end{theorem}

In this paper, we consider intersecting families that are not necessarily uniform. 
Let $J(n,k)$ denote the number of intersecting families $\mathcal{F} \subset 2^{[n]}$ with $\#F \le k$ for all $F \in \mathcal{F}$.
Our main result extends the classical counting results for uniform intersecting families to the non-uniform setting.

\begin{theorem} \label{thm:main}
    For $n \ge 2k+2+ 2\sqrt{k \log k}$ and $k \rightarrow +\infty$, 
    \[
        J(n,k) = (n + o(1)) 2^{\sum_{i=1}^k \binom{n-1}{i-1}}.
    \]
\end{theorem}

Theorem \ref{thm:main} implies that almost all intersecting families $\mathcal{F} \subset \binom{[n]}{\leq k}$ are trivial.
Here ``almost all'' means that the proportion of intersecting families that are not trivial tends to zero as $k$ tends to infinity. Theorem \ref{thm:main} shows that even when families of different sizes are allowed, the typical intersecting family is still trivial, provided $n$ is sufficiently large relative to $k$.

We remark that Theorem \ref{thm:main} cannot be derived directly from the known results on uniform intersecting families. Non-uniform families may contain sets of various sizes simultaneously, which brings in additional structural complications that do not arise in the uniform case. New arguments are therefore required beyond those used in the uniform setting.

\section{Notations and preliminaries}

%Before presenting the proof of Theorem \ref{thm:main}, let us introduce some notations and results needed.
Set 
\[
    \mathcal{F}(i) := \{F \backslash \{i\}:\, i\in F \in \mathcal{F} \}, \;
    \mathcal{F}(\overline{i}) := \{F :\, F \in \mathcal{F},\,  i\notin F \}.
\]
The estimate for the number of non-trivial uniform intersecting families will be needed.

\begin{theorem} [Frankl and Kupavskii \cite{FK18}] \label{thm:FK}
    Let $n$ and $k$ be integers with $n \ge 2k+2+ 2\sqrt{k \log k}$. 
    For $k \rightarrow +\infty$, 
    \begin{equation} \label{nieq1}
        I(n,k) = (n + o(1)) 2^{\binom{n-1}{k-1}},
    \end{equation}
    \begin{equation} \label{nieq2}
        I(n,k,\ge 1) = (1 + o(1))n \binom{n-1}{k} 2^{\binom{n-1}{k-1} - \binom{n-k-1}{k-1}}.
    \end{equation}
\end{theorem}

Two families $\mathcal{F} \subset \binom{[n]}{k}$ and $\mathcal{G} \subset \binom{[n]}{\ell}$ are called \textit{cross-intersecting} if $ A\cap B\neq \emptyset$
for every $A \in \mathcal{F}$ and $B \in \mathcal{G}$.

\begin{theorem}[Frankl and Tokushige \cite{FT92}]\label{thm-FT-92} 
    Let $n,k,\ell$ be integers with $n \ge k+\ell$ and $k \ge \ell$. 
    If $\mathcal{F} \subset \binom{[n]}{k}$ and $\mathcal{G} \subset \binom{[n]}{\ell}$ are cross-intersecting, then
    \begin{equation} \label{cieq}
        \#\mathcal{F} + \#\mathcal{G} \le \binom{n}{k} - \binom{n-\ell}{k} + 1.
    \end{equation}
\end{theorem}

The following result is clear from Theorem \ref{thm-FT-92}.

\begin{corollary}\label{corollary:F}
    Let $n,k,\ell$ be integers with $n \ge k+\ell$ and $k \ge \ell$. 
    If $\mathcal{F} \subset \binom{[n]}{k}$ and $\mathcal{G} \subset \binom{[n]}{\ell}$ are cross-intersecting and non-empty, then
    \[
        \#\mathcal{F} \le \binom{n}{k} - \binom{n-\ell}{k} .
    \]
\end{corollary}

\begin{remark}
The restriction 
$
n \ge 2k+2+2\sqrt{k\log k} 
$
in Theorem \ref{thm:main} is inherited from the current best known bounds in the uniform setting in Theorem \ref{thm:FK}. Any improvement of the corresponding range in Theorem \ref{thm:FK} would yield an analogous improvement for the non-uniform case considered here.
\end{remark}

\section{Proof of Theorem \ref{thm:main}}
%This section is devoted to proving Theorem \ref{thm:main}. 
The number of trivial intersecting families is $n2^{\sum_{i=1}^k \binom{n-1}{i-1}}$. 
It is sufficient to show that the number of non-trivial intersecting families is $o\left(2^{\sum_{i=1}^k \binom{n-1}{i-1}}\right)$.
Let $\mathcal{H}$ denote the collection of all non-trivial intersecting families $\mathcal{F} \subset \binom{[n]}{\le k}$. 
Set $\mathcal{F}_k:=\{F\in \mathcal{F}\colon \# F=k\}.$
Partition $\mathcal{H}$ into two subfamilies,
\[
    \mathcal{A} := \{ \mathcal{F} \in \mathcal{H}: \cap_{F \in \mathcal{F}_k}F = \emptyset \}, \;
    \mathcal{B} := \{ \mathcal{F} \in \mathcal{H}:\cap_{F \in \mathcal{F}_k}F \ne \emptyset\}.
\]
We show that both $\#\mathcal{A}$ and $\#\mathcal{B}$ are $o( 2^{\sum_{i=1}^k \binom{n-1}{i-1}})$.

\begin{lemma}\label{lem-1}
$\#\mathcal{A}=o( 2^{\sum_{i=1}^k \binom{n-1}{i-1}})$.
\end{lemma}

\begin{proof}
By $\mathcal{F}_k\subset \binom{[n]}{k}$ and Theorem \ref{thm:FK} (\ref{nieq2}), the number of choices for $\mathcal{F}_k$ is at most 
\[
(1 + o(1))n \binom{n-1}{k} 2^{\binom{n-1}{k-1} - \binom{n-k-1}{k-1}}.
\] 
For each $1\le i\le k-1$, the number of choices for $\mathcal{F}_i$ is at most $(n + o(1)) 2^{\binom{n-1}{i-1}}$ by Theorem \ref{thm:FK} (\ref{nieq1}). Since $\mathcal{F}=\bigcup_{i=1}^k \mathcal{F}_k$, we deduce that 
\begin{align*}
    \frac{\#\mathcal{A}}{\prod_{i=1}^{k}2^{\binom{n-1}{i-1}}} &\le \frac{(1 + o(1))n \binom{n-1}{k} 2^{\binom{n-1}{k-1} - \binom{n-k-1}{k-1}} \prod_{i=1}^{k-1} (n + o(1)) 2^{\binom{n-1}{i-1}} }{\prod_{i=1}^{k}2^{\binom{n-1}{i-1}}} \\
    & < \frac{2n \cdot \binom{n-1}{k} 2^{\binom{n-1}{k-1} - \binom{n-k-1}{k-1}} \prod_{i=1}^{k-1} 2n \cdot 2^{\binom{n-1}{i-1}} }{\prod_{i=1}^{k}2^{\binom{n-1}{i-1}}} \\
    & = 2^k n^k \binom{n-1}{k} 2^{-\binom{n-k-1}{k-1}}.
\end{align*}
Using $ \binom{n-1}{k} \le n^k = 2^{k \log n}$, we have
\begin{equation}
    \frac{\#\mathcal{A}}{\prod_{i=1}^{k}2^{\binom{n-1}{i-1}}}  < 2^{k+ 2k \log n - \binom{n-k-1}{k-1}}.
\end{equation}
Note that when $n\ge 2k+4$ and $k\ge 3$, 
\begin{equation}\label{ine-n28}
\binom{n-k-1}{k-1}\ge \binom{\frac{n}{2}+1}{k-1}\ge \binom{\frac{n}{2}+1}{2}\ge \frac{n^2}{8}.\end{equation}
Thus
$$
    \frac{\#\mathcal{A}}{\prod_{i=1}^{k}2^{\binom{n-1}{i-1}}}  \le 2^{k + 2k\log n  - \frac{n^2}{8}}.
$$
Since $n \ge 2k + 2 + 2 \sqrt{k\log k}$, then $k + 2k\log n - \frac{n^2}{8} \rightarrow -\infty$ when $k \rightarrow +\infty$, implying
$
    \#\mathcal{A} \le o( 2^{\sum_{i=1}^k \binom{n-1}{i-1}}).
$
\end{proof}

%We next give an estimation on $\#\mathcal{B}$.

\begin{lemma}\label{lem-2}
$\#\mathcal{B}=o( 2^{\sum_{i=1}^k \binom{n-1}{i-1}}).$
\end{lemma}

\begin{proof}
For any $\mathcal{F}\in\ \mathcal{B}$, we will give the upper bound of the possible choices of both $\mathcal{F}_k$ and $\mathcal{F}\setminus\mathcal{F}_k$.   For any $ \mathcal{F} \in \mathcal{B}$, by definition, the $k$-uniform intersecting family $\mathcal{F}_k$ is trivial. Thus we may assume $a\in\cap_{F\in\mathcal{F}_k} F$.

Since $\mathcal{F}\in\mathcal{B}\subset \mathcal{H}$ and using the definition of $\mathcal{H}$, then we see that $\cap_{F\in\mathcal{F}} F=\emptyset$, which means there exist $G \in \mathcal{F} \backslash \mathcal{F}_k$ with $ a \notin G$. Since $\mathcal{F}$ is intersecting, we get $G\cap F\ne\emptyset$ for any $F\in\mathcal{F}_k$. Note that $a\in F$ and $a\not\in G$, we get that $F\cap G\ne\emptyset$. Thus $\mathcal{F}_k(a)$ and $G$ are cross-intersecting. By $\# F'=k-1\ge \# G$ for all $F' \in \mathcal{F}_k(a)$ and Corollary \ref{corollary:F}, 
\begin{equation}\label{eq-up-B-Fk}
   \#\mathcal{F}_k = \#\mathcal{F}_k(a) \le \binom{n-1}{k-1} - \binom{n-i-1}{k-1} \le \binom{n-1}{k-1} - \binom{n-k-1}{k-1},
\end{equation}
where $\# G=i\le k-1$.

Moreover, for fixed $a$ and fixed $\mathcal{F}_1,\mathcal{F}_2, \ldots, \mathcal{F}_{k-1}$, let 
$$
\mathcal{K}
:=\left\{F\in \binom{[n]}{k}\colon a\in F,\ F\cap H\ne\emptyset \text{ for any }H\in\mathcal{F}_1\cup\mathcal{F}_2\cup \cdots \cup  \mathcal{F}_{k-1}\right\}.
$$
Clearly, for any $1\le i\le k-1$, $\mathcal{K}$ and $\mathcal{F}_i$ are crossing intersecting. Notice that $a\in F$ for any $F\in \mathcal{K}$, which implies that $\mathcal{K}$ itself is an intersecting family. Thus 
$$\mathcal{I}:=\mathcal{K}\cup\mathcal{F}_1\cup\mathcal{F}_2\cup \dots \cup \mathcal{F}_{k-1}$$
 is an intersecting family. From \eqref{eq-up-B-Fk}, we deduce that $\#\mathcal{K}\le \binom{n-1}{k-1} - \binom{n-k-1}{k-1}$. Since $\mathcal{F}_k$ is a subset of $\mathcal{K}$ and $G \in \cup_{i=1}^{k-1} \mathcal{F}_i$, we see that for fixed $a$ and fixed $\mathcal{F}_1,\mathcal{F}_2, \ldots, \mathcal{F}_{k-1}$, there are at most $2^{\binom{n-1}{k-1} - \binom{n-k-1}{k-1}}$ ways to choose $\mathcal{F}_k$. By (\ref{nieq1}), the number of intersecting families $\mathcal{F}_i$, $1\le i\le k-1$, is at most $(n + o(1)) 2^{\binom{n-1}{i-1}}$. Together with the $n$ choices of $a$ and $\sum_{i=1}^{k-1} \binom{n-1}{i}$ choices of $G$,  we get
\begin{align*}
    \#\mathcal{B}
    &\le n\cdot 2^{\binom{n-1}{k-1} - \binom{n-k-1}{k-1}}\cdot \sum_{i=1}^{k-1} \binom{n-1}{i} \cdot \prod_{i=1}^{k-1}(n + o(1)) 2^{\binom{n-1}{i-1}}\\
    &<  n 2^{- \binom{n-k-1}{k-1}} 2^n (2n)^{k-1} \cdot \prod_{i=1}^{k} 2^{\binom{n-1}{i-1}}   \\
    &=  2^{n + k-1}n^k 2^{-\binom{n-k-1}{k-1}} \cdot \prod_{i=1}^{k} 2^{\binom{n-1}{i-1}} \\
    &= 2^{n+ k -1+ k \log n - \binom{n-k-1}{k-1}}\cdot \prod_{i=1}^{k} 2^{\binom{n-1}{i-1}} .
\end{align*}
  From \eqref{ine-n28}, when $n>2k+4$ and $k \rightarrow + \infty$, 
  \[n+ k-1 + k \log n - \binom{n-k-1}{k-1} \le n+ k+k\log n-\frac{n^2}{8}\rightarrow - \infty.\]
  Thus
$
    \#\mathcal{B}= o(2^{\sum_{i=1}^k \binom{n-1}{i-1}}).
$
\end{proof}

Since $\#\mathcal{H}=\#\mathcal{A}+\#\mathcal{B} = o(2^{\sum_{i=1}^k \binom{n-1}{i-1}})$, the proof is complete.

%%%%%%%%%%%%%%
\section*{Acknowledgement}
%%%%%%%%%%%%%%

The research of Yandong Bai and Haoyun Gu was supported in part by the National Key Research and Development Program of China (Grant No. 2026YFE0151700), the National Natural Science Foundation of China (Grant Nos. 12542043, 12242111, 12131013) and Guangdong Basic and Applied Basic Research Foundation (Grant No. 2023A1515030208). 
The research of Wenston J.T. Zang was
supported by the National Natural Science Foundation of China (Grant No. 12371336).

\end{document}